\documentclass{amsart}

\usepackage[utf8]{inputenc}
\usepackage{amsmath,amssymb,amsthm,units, stmaryrd,stackrel,relsize,bm}
\usepackage{enumitem}
\usepackage{mathdots}
\usepackage{bussproofs}
\usepackage{tikz,tikz-cd}
\usetikzlibrary{calc,arrows, decorations.pathmorphing}

\newtheorem{theorem}{Theorem}

\newtheorem{lemma}[theorem]{Lemma}

\newtheorem{claim}[theorem]{Claim}
\newtheorem{corollary}[theorem]{Corollary}

\theoremstyle{definition}
\newtheorem{definition}[theorem]{Definition}
\theoremstyle{remark}
\newtheorem{remark}[theorem]{Remark}
\usepackage[
    colorlinks=true, 
    citecolor=red,
    linkcolor=blue,
    urlcolor=blue]{hyperref}

\newcommand\zfc{{\mathsf{ZFC}}}

\newcommand\kp{{\mathsf{KP}}}

\newcommand\Ord{{\mathsf{Ord}}}

\newcommand\dom{{\text{dom}}}

\newcommand\wfp{\mathrm{wfp}}

\newcommand\wo{{\textsc{wo}}}

\newcommand\ca{{\mathsf{CA_0}}}
\newcommand\atr{{\mathsf{ATR_0}}}
\newcommand\aca{{\mathsf{ACA_0}}}
\newcommand\rca{{\mathsf{RCA_0}}}

\newcommand\supp{\mathrm{supp}}

\newcommand{\PI}{\boldsymbol\Pi}
\newcommand{\SIGMA}{\boldsymbol\Sigma}
\newcommand{\DELTA}{\boldsymbol\Delta}

\makeatletter
\def\Ddots{\mathinner{\mkern1mu\raise\p@
\vbox{\kern7\p@\hbox{.}}\mkern2mu
\raise4\p@\hbox{.}\mkern2mu\raise7\p@\hbox{.}\mkern1mu}}
\makeatother

\usepackage{etoolbox}
\makeatletter
\pretocmd{\chapter}{\addtocontents{toc}{\protect\addvspace{13\p@}}}{}{}
\pretocmd{\section}{\addtocontents{toc}{\protect\addvspace{9\p@}}}{}{}
\pretocmd{\subsection}{\addtocontents{toc}{\protect\addvspace{2\p@}}}{}{}
\makeatother
\makeatletter
\def\@tocline#1#2#3#4#5#6#7{\relax
  \ifnum #1>\c@tocdepth 
  \else
    \par \addpenalty\@secpenalty\addvspace{#2}%
    \begingroup \hyphenpenalty\@M
    \@ifempty{#4}{%
      \@tempdima\csname r@tocindent\number#1\endcsname\relax
    }{%
      \@tempdima#4\relax
    }%
    \parindent\z@ \leftskip#3\relax \advance\leftskip\@tempdima\relax
    \rightskip\@pnumwidth plus4em \parfillskip-\@pnumwidth
    #5\leavevmode\hskip-\@tempdima #6\nobreak\relax
    \ifnum#1<0\hfill\else\dotfill\fi\hbox to\@pnumwidth{\@tocpagenum{#7}}\par
    \nobreak
    \endgroup
  \fi}
\makeatother
\begin{document}
\title[Reverse Mathematics of Analytic Measurability]{The Reverse Mathematics of\\ Analytic Measurability}
\subjclass[2020]{03B30 (primary), 03D60, 28A05 (secondary)}
\keywords{reverse mathematics, analytic set, Lebesgue measure, hyperarithmetical set}
\author{J. P. Aguilera, T. Kouptchinsky, and K. Yokoyama}

\begin{abstract}
It is shown that the Lebesgue-regularity of analytic sets is equivalent to $\SIGMA^1_1$-Induction over $\atr$. On the other hand, the Lebesgue-measurability of analytic sets is equivalent to $\PI^1_1{-}\mathsf{CA}_0$. This answers a question of Simpson (1999).
\end{abstract}
\date{\today $\,$ (compiled)}
\clearpage
\maketitle
\numberwithin{equation}{section}
\setcounter{tocdepth}{1}

\section{Introduction}\label{SectionIntro}
A set $A\subset [0,1]$ is called \textit{analytic} if it is a continuous image of a Borel set. A classical theorem of Lusin \cite{Lu17} is that all analytic sets are Lebesgue-measurable. 
Reverse Mathematics is the branch of Mathematical Logic which studies the question: ``given a mathematical theorem $\phi$, which axioms are necessary to prove $\phi$?''.  This article concerns the Reverse Mathematics of Lusin's theorem. We shall see that the strength of Lusin's theorem depends on how precisely it is formalized.

We work throughout in the framework of second-order arithmetic. 
Real numbers are each identified with a set of natural numbers $X \subset \mathbb{N}$; $\phi(X)$ in the language of second-order arithmetic define sets of reals. 
We write $X \in A_{\phi}$ if $\phi(X)$ holds. Often we simply write $X \in A$ when referring to a set $A$ and omit its defining formula. 
Friedman's \cite{Fr75} system of \textit{Arithmetical Transfinite Recursion} ($\atr$) yields a convenient framework for working with Borel sets and developing measure theory over these.
 
A set of reals $A$ is \textit{Lebesgue-regular} if
$\lambda^*(A) = \lambda_*(A),$
where
\[\lambda^*(A) = \inf\Big\{\lambda(B) : \text{$A\subset B$ and $B$ is open}\Big\}\]
is the \textit{outer Lebesgue measure} of $A$ and 
\[\lambda_*(A) = \sup\Big\{\lambda(B) : \text{$B\subset A$ and $B$ is compact}\Big\}\]
is the \textit{inner Lebesgue measure} of $A$. If $\lambda^*(A) = \lambda_*(A)$ and this value exists, we denote it by $\lambda(A)$ and say that $A$ is \textit{Lebesgue measurable}. When working in strong theories such as $\zfc$, the regularity and measurability of sets is seen to be equivalent and has many other characterizations familiar from measure theory, but this is not immediately clear in the context of subsystems of analysis.

The Reverse Mathematics of measure theory has been developed by various authors, particularly by Simpson and Yu \cite{SY90} and of Brown, Giusto, and Simpson \cite{BGS02}. As far as recent results are concerned, we mention the work of Westrick \cite{MR4900338}, Astor, Dzhafarov, Montalb\'an, Solomon, and Westrick \cite{MR4085059}.
We also mention the work of Yu \cite{Yu90a}, Avigad, Dean, and Rute \cite{ADR12} and of Kjos-Hanssen, Miller, and Solomon \cite{KHMiSo12}.

The existence of $\lambda(U)$ for all open sets $U$ is equivalent to Arithmetical Comprehension ($\aca$) over Recursive Comprehension ($\rca$). $\atr$ is sufficient to prove that $\lambda(A)$ exists for all Borel sets $A$, by a theorem of Yu \cite{Yu93}. In his textbook on Reverse Mathematics, Simpson \cite{Simpson} asked whether $\atr$ ``suffices to prove measurability and regularity of analytic sets in some appropriate sense.'' The purpose of this article is to answer this question. The following is our main theorem:

\begin{theorem}\label{TheoremRegular}
The following are equivalent over $\atr$.
\begin{enumerate}
\item All analytic sets $A \subset [0,1]$ are Lebesgue-regular; and
\item $\SIGMA^1_1$-Induction for $\mathbb{N}$.
\end{enumerate}
\end{theorem}

In addition to answering Simpson's question, Theorem \ref{TheoremRegular} is remarkable in that it presents a reversal of a theorem to $\SIGMA^1_1$-Induction, an axiom that essentially deals with $\mathbb{N}$. Nemoto \cite{Ne09} has shown that $\SIGMA^1_1{-}\mathsf{IND}$ is equivalent to a determinacy principle for Cantor space over $\mathsf{ATR}_0$, namely, the determinacy of all games with payoff in the class $B(\DELTA^0_1,\SIGMA^0_2)$ of all two-sided separated unions of $\SIGMA^0_2$ and $\PI^0_2$ by clopen sets. Day, Greenberg, Harrison-Trainor, and Turetsky \cite{DGHTT} have also made use of $\SIGMA^1_1{-}\mathsf{IND}$ in their recent work on the effective classification of Borel Wadge classes. We also mention recent work by Towsner, Weisshaar, and Westrick \cite{MR4575273}, Fern\'andez-Duque, Shafer, Towsner, and Yokoyama \cite{MR4590512}, and Fern\'andez-Duque, Shafer, and Yokoyama \cite{MR4160946}, which deal with the reverse mathematics of theorems in the region surrounding $\atr$ and $\Pi^1_1{-}\mathsf{CA}_0$.

Theorem \ref{TheoremRegular} deals with Lebesgue-regularity, i.e., the equality of the outer and inner measures. If one additionally demands that this value exist as a number, the strength of Lusin's theorem increases:

\begin{theorem}\label{TheoremMeasurable}
The following are equivalent over $\atr$.
\begin{enumerate}
\item All analytic sets $A \subset [0,1]$ are Lebesgue-measurable; and
\item $\PI^1_1$-Comprehension.
\end{enumerate}
\end{theorem}

Theorem \ref{TheoremMeasurable} is much easier to establish than Theorem \ref{TheoremRegular}, but it clarifies the picture and we include it for completeness. The fact that the space $[0,1]$ is of finite measure or compact does not play a role in the theorems; see \S\ref{SectConclusion}.

Typically, in Reverse Mathematics the bulk of the work consists in reversing the theorems to the axioms. Here, this is arguably not the case. In \S\ref{SectReversals} we present the reversals of Theorem \ref{TheoremRegular} and Theorem \ref{TheoremMeasurable}. 
Most of the work however lies in the proof of Lebesgue-regularity from $\SIGMA^1_1$-Induction alone. For this proof, the most delicate step occurs in \S\ref{SectionApprox}, where the hypothesis of $\SIGMA^1_1$-Induction on $\mathbb{N}$ is used. 
The purpose of \S\ref{SectModelM}--\S\ref{SectSG} is to set the stage for \S\ref{SectionApprox}.
The main idea is to draw inspiration from Solovay's \cite{So70} construction of a model of Zermelo-Fraenkel set theory where every set is Lebesgue measurable. In our case the argument requires the use of class forcing over a family of standard and non-standard models of a very weak set theory obtained through the familiar method of \textit{pseudohierarchies}. Forcing over the standard models yields a sequence of approximations to the inner measure of a given co-analytic set $A$, and forcing over the non-standard models yields a sequence of approximations to its outer measure. The hypothesis of $\SIGMA^1_1$-Induction is used in showing that these approximations eventually become arbitrarily good, proving the theorem.

\subsection*{Acknowledgements}
The work of J. P. A. and T. K.  was partially supported by FWF grants P36837, STA-139, and PAT2264325.
The work of K. Y. was partially supported by  JSPS KAKENHI grant numbers JP21KK0045 and JP23K03193. The first author is grateful to the other two authors for saving his life in the Japanese wilderness, and even more grateful to the Erwin Schr\"odinger Institute for their support as part of the thematic program ``Reverse Mathematics'' in 2025. The authors would like to thank T. Nemoto, A. Weiermann, and L. Yu for comments on talks given on this work. This research benefited from discussions at the RIMS symposium
"Computability and Complexity in Analysis 2025."

\section{Preliminaries and notation}
The theory $\atr$ in the two-sorted language $\mathrm{L}_2 = \{+,\times, \in\}$ of second-order arithmetic consists of the Robinson axioms $Q$ for addition and multiplication together with the induction axiom 
\[0 \in X \wedge \forall x\, (x \in X \to x+1 \in X) \to \forall x\, (x\in X)\]
and the axiom schema of arithmetical transfinite recursion:
\[\forall {\prec} \Big(\text{$\prec$ is a wellorder } \to \exists X\, \forall a\in\dom(\prec)\,\forall x\, \big(x\in X_{a}\leftrightarrow \varphi(x, X_{<a})\big) \Big),\]
where $\varphi$ is an arithmetical formula, i.e., a formula without second-order quantifiers, possibly with free variables. Here,
$X_a = X \cap \{a\}\times\mathbb{N}$
is the $a$th slice of $X$, and $X_{<a}$ is defined by $X_{<a} = \{(b,m) \in X: b \prec a\}$. For further background on this and Reverse Mathematics in general, we refer the reader to Simpson \cite{Simpson}. 

Below, we shall generally work in $\atr$, although at various points it will be necessary to prove lemmata in weaker theories. For the sake of completeness, let us recall the definitions of the theories which will appear below. When speaking of classes of formulas such as $\SIGMA^1_1$, we use the boldface notation to emphasize that free second-order variables are allowed, in contrast with the lightface notation $\Sigma^1_1$ which does not allow these.
\begin{enumerate}
\item $\rca$ is the base theory of \textit{Recursive Comprehension}, consisting of $Q$ together with the induction schema for $\SIGMA^0_1$ formulas, and the $\DELTA^0_1$-comprehension schema,
\[\forall x\, (\phi(x) \leftrightarrow  \lnot \psi(x)) \to \exists X\, \forall x\, (x \in X \leftrightarrow \phi(x)),\]
where $\phi$ and $\psi$ are $\SIGMA^0_1$-formulas. 
\item $\mathsf{WKL}_0$ is \textit{Weak K\"onig's Lemma}, extending $\rca$ with the axiom asserting that every infinite binary tree $T\subset 2^{<\mathbb{N}}$ has a(n infinite) branch.
\item $\aca$ is \textit{Arithmetical Comprehension}, extending $\rca$ with the comprehension schema
\[\exists X\, \forall x\, (x\in X\leftrightarrow \phi(x))\]
 for arithmetical formulas.
\item $\PI^1_1{-}\mathsf{CA}_0$ is $\PI^1_1$-Comprehension, extending $\rca$ with the comprehension schema for $\PI^1_1$-formulas.
\end{enumerate}

\subsection{Models of set theory}
The theory $\atr$ allows us to represent models of set theory in second-order arithmetic via so-called \textit{suitable trees} $T \subset \mathbb{N}^{< \mathbb{N}}$. These trees use the edge relation to represent membersip between transitive sets. $\atr$ provides a meaningful interpretation of these \textit{suitable trees} as sets, a translation $|\phi| \in \mathrm{L}_2$ for any sentence $\phi \in \mathrm{L}_{\mathrm{set}}$ and a theory $\atr^{\mathrm{set}}$ so that $\atr^{\mathrm{set}} \vdash \phi$ if and only if $\atr \vdash |\phi|$. We refer the reader to Simpson~\cite[VII.3]{Simpson} for details. 
In what follows we shall move back and forth through this translation freely without mention and identify $\atr^{\mathrm{set}}$ with $\atr$. When speaking explicitly of models of arithmetic, we generally attempt to use lowercase variables to refer to first-order objects (i.e., numbers) and uppercase variables to refer to second-order objects (i.e., sets of numbers).

Throughout, we work implicitly in some model $V$ of $\atr$ and make use of other models coded by sets $M \in V$. We say that $M$ is an \textit{$\omega$-model} if $\mathbb{N}^M$ (the set of natural numbers of $M$)  is equal to $\mathbb{N}$ ($=\mathbb{N}^V$). We sometimes write $V\models\phi$ to emphasize that we are concerned with the truth value of the formula $\phi$ in $V$ and not in some other model.

We work with nonstandard models of set theory. If $M$ is such a model, we use Greek letters $\alpha$, $\beta$, $\gamma$, etc.~for $M$-ordinals which are wellfounded, and Latin letters $a,b,c,$ etc.~ for $M$-ordinals which might be illfounded.

If $M$ is a model of $V = L$, then $M$ believes that it satisfies the Axiom of Global Choice, as witnessed by a $\Sigma_1$-wellordering of the universe. This wellordering is defined by the same formula in all models of $V = L$ and is given by G\"odel's proof that the Axiom of Choice holds in $L$; we denote it by $<_M$. If $M$ is illfounded, then $<_M$ will be an illfounded order, however. Nonetheless, if $M$ satisfies the schema of \textsc{foundation}, $<_M$ will have no ``simple'' infinite descending sequences, that is, none which exist as sets in $M$.

\subsection{Wellorders} 
We use $\wo$ ambiguously for the class of \textit{recursive} subsets of $\mathbb{N}$ which code wellorders, as well as for the set of indices of such recursive sets (field of the wellfounded relation on $\mathbb{N}\times \mathbb{N}$).
Note in particular that this deviates from common usage of $\wo$, which typically refers to the class of \textit{all} wellordered relations on $\mathbb{N}$. Nonetheless, we will need to consider the effective version of $\wo$ for a more careful analysis.
The relativization of this class to a real parameter $X$ is denoted $\wo^X$. The supremum of lengths of elements of $\wo^X$ is denoted by $\omega_1^X$. Typically, one writes $\omega_1^{ck}$ (for Church-Kleene) instead of $\omega_1^0$. Note, however, that $\atr$ is not powerful enough to prove the existence of $\omega_1^{ck}$.

\subsection{Measure theory and Borel sets}
We denote by $C[0,1]$ the separable Banach space of all continuous functions on the unit interval. 
The Lebesgue measure of a basic open set $(a,b)$ is defined as $b-a$. This is then extended to arbitrary Borel sets using the Carath\'eodory extension theorem (see e.g., Halmos \cite{Ha74}).

Borel subsets of $[0,1]$ are identified with \textit{Borel codes}. These are well-founded trees $T \subset \mathbb{N}^{<\mathbb{N}}$ satisfying the following:  
\begin{enumerate}
  \item The leaf nodes are labelled with basic open sets;
  \item A node $\sigma$ represents the (possibly infinite) union of the complements of every set represented by $\sigma^{\smallfrown}n$ for $n \in \mathbb{N}$.
\end{enumerate}

Working in $\atr$, we can speak of Borel sets through their codes.
For such a code $T$, $\mathsf{ATR}_0$ guarantees the existence of a unique evaluation map $f_T$ interpreting the tree, giving rise to a formula $\phi_T \in \mathrm{L}_2$ so that $\phi_T(X)$ holds whenever $X$ belongs to the Borel set. $\mathsf{ATR}_0$ is powerful enough to prove that Borel codes behave as one would expect.

Note that a given code may have different interpretations in different models of set theory (or of second-order arithmetic), even if they have the same set of natural numbers. We use the notation $c_A$ ambiguously to denote some tree corresponding to a Borel set $A$. We write ``$X \in A$'' for $X \subset 2^{\mathbb{N}}$ to say that the evaluation function for $c_A$ at $X$ witnesses $X$ to be an element of $A$. This is technically an abuse of notation, as there might be several codes for the same set, but this shall cause no problems in what follows.

We use $\lambda^*$ and $\lambda_*$ to denote the inner and outer Lebesgue measures. For arbitrary sets $A$ such that $\lambda^*(A) = \lambda_*(A)$, we denote this common value by $\lambda(A)$.
The theory of these operations was developed in $\rca$ by Yu~\cite{Yu93}. Indeed, it is shown in \cite{Yu93} that the existence of  $\lambda(A)$ for open and closed $A$ (as real numbers) relies on $\aca$, as well as on $\atr$ for arbitrary Borel sets.  A set $A$ is said to have \textit{measure zero} if $\lambda^*(A) = 0$. We shall need the following two lemmata:
\begin{lemma} \label{LemmaMeasureZeroSigma2}
Assume $\mathsf{WWKL}_0$. Suppose $A$ and $B$ are $\SIGMA^0_2$ subsets of $[0,1]$. Then, the formula $\lambda(A\setminus B) = 0$ is arithmetical.
\end{lemma}
\proof
Without defining $\mathsf{WWKL}_0$, we use the fact that this theory is able to prove the $\sigma$-additivity of the Lebesgue measure on arithmetical sets, by a theorem of Simpson and Yu \cite{SY90}.
Let us first show that there is an arithmetical formula $\psi(c_U, n)$ expressing $\lambda(U) < 1/n$ whenever $c_U$ is a code for an open set $U$.

Write $U = \bigcup_{i\in\mathbb{N}} (a_i,b_i)$ as a union of basic open sets provided by $c_U$. From this representation, we can obtain a possibly different Borel code of $U$ as a union of disjoint basic open intervals $(a_i', b_i')$ and this can be done in a recursive way, uniformly for all codes $c_U$. But then
\[\lambda(U) = \sum_{i = 0}^\infty b_i' - a_i'\]
(using $\mathsf{WWKL}_0$),
from which the claim follows. 

Now to prove the lemma, let $A$ be $\SIGMA^0_2$ and $B$ be $\PI^0_2$ and write $A = \bigcup_{i\in\mathbb{N}} A_i$  with each $A_i$ closed and increasing in $i$. Then,
\[\lambda(A\cap B) = \lambda\Big(\bigcup_{i\in\mathbb{N}}A_i \cap B\Big) = \lambda\Big(\bigcup_{i\in\mathbb{N}}(A_i \cap B)\Big) = \lim_{i \to \infty} \lambda(A_i\cap B).\]
As each set $A_i \cap B$ is $\PI^0_2$, it suffices to verify that the formula $\lambda(B) = 0$ is arithmetical when $B$ is $\PI^0_2$ (this is once more by $\mathsf{WWKL}_0$). 

$B$ be $\PI^0_2$ and write  $B = \bigcap_{i\in\mathbb{N}} B_i$ with the sets $B_i$ open and decreasing in $i$. Then, $\lambda(B) = 0$ if and only if for all $n\in\mathbb{N}$ there is $i$ such that $\lambda(B_i) < 1/n$. This latter formula is arithmetical by the initial claim, from which the lemma follows.
\endproof

\begin{lemma}\label{LemmaEmptyClosedSets}
Assume $\mathsf{WKL}_0$. Suppose $A\subset [0,1]$ is $\SIGMA^0_2$. Then, the formula $A = \varnothing$ is arithmetical.
\end{lemma}
\proof
Since $A$ is $\SIGMA^0_2$, we can write $A  = \bigcup_{i\in\mathbb{N}} A_i$ for some sequence of closed sets $A_i$, and $A$ is empty if and only if each $A_i$ is empty.

It thus suffices to consider the formula $C = \varnothing$ for closed codes $c_C$. Write $C = \bigcap_{i\in\mathbb{N}} C_i$ as an intersection of basic closed sets of one of the forms $[0,a_i] \cup [b_i, 1]$, $[0,a_i]$, or $[b_i,1]$. Let $T$ be the tree obtained as follows: nodes $p \in T$ are intervals $[c_p, d_p]$, with the root being the unit interval $[0,1]$. Letting $i$ be the height of $p$, we put $q$ above $p$ in $T$ if $q$ is a nonempty interval of one of the forms
\[p \cap [0,a_{i+1}] \text{ or } p \cap [b_{i+1}, 1],\]
where $a_{i+1}$ and $b_{i+1}$ are obtained from $C_{i+1}$ as above. Since $\mathsf{WKL}_0$ holds, the intersection $C = \bigcap_{i\in\mathbb{N}}$ is nonempty if and only if there is $T$ is illfounded, and again by $\mathsf{WKL}_0$ this holds if and only if $T$ is infinite. Clearly, $T$ is arithmetically definable from the code of $C$ and the assertion that $T$ is infinite is $\Pi^0_2(T)$.
\endproof

\section{Reversals} \label{SectReversals}
Towards establishing the theorems in the introduction, we begin with the two reversals.
We will use the following characterization of induction.
\begin{lemma}\label{LemmaInductionComprehension}
The following are equivalent over $\rca$:
\begin{enumerate}
\item\label{LemmaInductionComprehension1} $\SIGMA^1_1$-Induction,
\item\label{LemmaInductionComprehension2} Finite $\SIGMA^1_1$-Comprehension, i.e., for any $\SIGMA^1_1$ formula $\phi$ and for all $n\in\mathbb{N}$, there exists a set $X \subset \mathbb{N}$ such that 
\[\forall m\leq n\, \big(m \in X\leftrightarrow \phi(m)\big).\]
\end{enumerate}
\end{lemma}
\proof
We reason within $\rca$.
It is easy to see \eqref{LemmaInductionComprehension2} implies \eqref{LemmaInductionComprehension1}. We show \eqref{LemmaInductionComprehension1} implies \eqref{LemmaInductionComprehension2}.
Let $\phi(m)\equiv \exists Z\theta(m,Z)$ be a $\SIGMA^{1}_{1}$-formula, and let $n\in\mathbb{N}$.
Consider the following $\SIGMA^{1}_{1}$-formula:
\[\psi(k)\equiv \exists Y\exists X\subseteq\{0,\dots,n\}(\forall m\in X \theta(m,Y_{m})\wedge k\leq |X|)\]
where $Y_{m}=\{x: (x,m)\in Y\}$ and $|X|$ denotes the cardinality of $X$.
It is easy to check that $\psi(0)\wedge \lnot\psi(n+1)$ holds.
Thus, by $\SIGMA^1_1$-Induction, there exists $k_{0}\leq n$ such that $\psi(k_{0})\wedge \neg\psi(k_{0}+1)$.
Take $Z$ and $X$ for $\psi(k_{0})$, then we see that this $X$ is the desired set. Indeed, if $m_{0}\notin X$ and $\exists Z\theta(m_{0},Z)$ for some $m_{0}\leq n$, then $\tilde Y$ defined by $\tilde Y_{m_{0}}=Z$ and $\tilde Y_{j}=Y_{j}$ for $j\neq m_{0}$ and $X\cup \{m_{0}\}$ witness $\psi(k_{0}+1)$.
\endproof

\begin{lemma}\label{LemmaReversalInd}
The following is provable within $\aca$.
Suppose that all analytic sets are Lebesgue-regular. Then, $\SIGMA^1_1$-Induction holds.
\end{lemma}
\proof
The rough idea for the proof is to consider the interval $[0,\mathcal{O}]$, where $\mathcal{O}$ is Kleene's set of ordinal notations. This appears to be  a closed set, but a code for it cannot be produced without access to $\mathcal{O}$. Nonetheless, one can show in weak theories that it is a co-analytic set. Intuitively, knowing the measure of $[0,\mathcal{O}]$ amounts to computing $\mathcal{O}$, and we shall argue that the Lebesgue-regularity of this set implies lightface $\Sigma^1_1{-}\mathsf{IND}$, after which a relativization argument will yield the desired result. We do so in more detail now.

We reason within $\aca$.
By Lemma \ref{LemmaInductionComprehension}, it suffices to establish Finite $\SIGMA^1_1$-Comprehension.
Let $\phi(m)\equiv\exists Z\theta(m,Z)$ be a $\SIGMA^{1}_{1}$-formula, and let $n\in\mathbb{N}$.
Define an analytic set $A\subseteq [0,1]$ so that
\[ x\in A \leftrightarrow \exists Y\exists X\subseteq \{0,\dots,n\}\left(\forall i\in X \,\theta(i,Y_{i}) \wedge 0\leq x\leq \sum_{i\in X}2^{-2i-1}\right).\]
Note that $A$ is ``interval-like'', in the sense that $[0,x]\subseteq A$ whenever $x \in A$.
By the regularity of $A$, take a compact set $G$ and an open set $H\subseteq [0,1]$ such that $G\subseteq A\subseteq H$ and $|\lambda (G)-\lambda(H)|<2^{-2n-3}$.
Put $\alpha=\max\{0,\lambda(G)-2^{-2n-4}\}$, then, we have $\alpha\in A$ and $\alpha+2^{-2n-2}\notin A$.
(If $\alpha\notin A$, then $[0,\alpha)\supseteq A\supseteq G$, which contradicts $\alpha<\lambda(G)$. Similarly, if $\alpha+2^{-2n-2}\in A$, then $[0,\alpha+2^{-2n-2}]\subseteq A\subseteq H$, which contradicts $\lambda(H)<\lambda(G)+2^{-2n-3}=\alpha+2^{-2n-4}+2^{-2n-3}$.)
Since $\alpha\in A$, there exist a set $Y\subseteq\mathbb{N}$ and a finite set $X\subseteq\{0,\dots,n\}$ such that $\forall i\in X \,\theta(i,Y_{i})$ and $\alpha\leq \sum_{i\in X}2^{-2i-1}$.
We see that $\forall i\leq n(i\in X\leftrightarrow \exists Z\theta(i,Z))$.
We only need to show the right-to-left implication. Assume that $i\notin X$, but $\theta(i,Z)$ for some $i\leq n$ and $Z$.
Define $\tilde X$ and $\Tilde Y$ to be $\tilde X=X\cup\{i\}$, and $\tilde Y_{i}=Z$ and $\tilde Y_{j}=Y_{j}$ if $j\neq i$.
Then $\alpha+2^{-2i-1}\in A$ witnessed by $\tilde X$ and $\tilde Y$, but this contradicts $\alpha+2^{-2n-2}\notin A$.
\endproof

\begin{lemma}\label{LemmaReversalComp}
The following is provable within $\aca$.
Suppose that all analytic sets are Lebesgue-measurable. Then, $\PI^1_1$-Comprehension holds.
\end{lemma}
\proof
We reason within $\aca$.
By Lemma \ref{LemmaReversalInd}, we may also assume $\SIGMA^1_1$-Induction.
Let $\phi(m)\equiv\exists Z\theta(m,Z)$ be a $\SIGMA^{1}_{1}$-formula.
Define an analytic set $A\subseteq [0,1]$ so that
\[ x\in A \leftrightarrow \exists Y\exists X\subseteq \mathbb{N}\left(\forall i\in X \,\theta(i,Y_{i}) \wedge 0\leq x\leq \sum_{i\in X}2^{-2i-1} \right).\]
By the measurability of $A$, let $\alpha=\lambda(A)$.
Using primitive recursion, we define a set $X$ by
\[i\in X \leftrightarrow \sum_{j\in X, j<i}2^{-2j-1}+2^{-2i-1}\le \alpha.\]
We see that $\forall i(i\in X\leftrightarrow \exists Z\theta(i,Z))$ by $\SIGMA^1_1$-Induction.
Let $i\in\mathbb{N}$ and assume that $j\in X\leftrightarrow \exists Z\theta(j,Z)$ for all $j<i$. Then by Finite $\SIGMA^1_1$-Comprehension, there exists $Y$ such that for all $j\in X$ with $j<i$, $\theta(j,Y_{j})$. Hence $\sum_{j\in X, j<i}2^{-2j-1}+2^{-2i-1}\le \alpha$ if and only if $\exists Z\theta(i,Z)$.
\endproof

What remains now are the proofs of Lebesgue regularity and Lebesgue measurability. We turn to these now.

\section{Proof of regularity}\label{SectRegularity}
We assume $\atr$ throughout. 
Let $A\in \Pi^1_1$ be a set of reals; our goal is to prove that $A$ is Lebesgue-regular. 
The proof will relativize to sets which are $\Pi^1_1$ relative to a real parameter.
The first step is to consider a non-standard model $M$ which will serve as a guide for the proof.

\subsection{The model $M$}\label{SectModelM}
We call Basic Set Theory ($\mathsf{BST}$) the set of axioms in the language of set theory consisting of \textsc{extensionality}, \textsc{pairing}, \textsc{union}, \textsc{infinity}, the scheme of $\Delta_0$-\textsc{separation}, and the scheme of \textsc{foundation} for all formulae of set theory. That is, $\kp \setminus \{\Delta_0$-\textsc{collection}$\}$. Moreover, we call ``$V=L$'' the axiom \[\forall \alpha \in \Ord \, \exists x \, (x = L_{\alpha}) \wedge \forall y\, \exists \alpha\in\Ord\, (y \in L_\alpha).\] 

If $M$ is an $\omega$--model of $\mathsf{BST}+$``$V = L$,'' then $M \cap \mathcal{P}(\mathbb{N}) \models \aca$ and $M$ is correct about arithmetical statements, including in particular those given by Lemma \ref{LemmaMeasureZeroSigma2} and Lemma \ref{LemmaEmptyClosedSets}.
$M$ might, however, contain some non-standard ordinals (externally ill-founded ordinals) and  we will in fact use this in our proof of analytic measurability. We shall make use of a particular $M$ below.

Recall Ville's lemma which asserts that the wellfounded part of any non-standard admissible set is admissible. This lemma might be vacuous over $\atr$, as there might not be any non-standard admissible sets and, even if there are, we might not be able to separate their wellfounded parts from their illfounded parts. 
The following result will serve as a proxy for Ville's lemma. 

Below, recall that $\wo$ is the set of \textit{recursive} wellorderings of $\mathbb{N}$. The reason for this choice of notation will become apparent later.

\begin{lemma}\label{LemmaNonstandardOrdinals}
There exists an $\omega$-model $M$ of $\mathsf{BST}$ satisfying $V = L$ and such that $M$ contains ordinals isomorphic to every element of $\wo$.
\end{lemma}
\proof 
This lemma is essentially the well-known fact that there exists a non-standard model of set theory with ordinal wellfounded part $\omega_1^{ck}$. However, the lemma is stated this way as  our base theory does not allow proving the existence of $\omega_1^{ck}$. 

Let $\phi(a)$ be the $\Sigma^1_1$ sentence formalizing the conjunction of $\textsc{lo}(a)$ and ``there is a countably coded $\omega$-model $M \models$ `$\mathsf{BST} + V=L$' such that $\Ord^M \cong a$.''
Using $\atr$, we can show that every member of $\wo$ satisfies $\phi$. However, being wellfounded is $\Pi^1_1$-complete (provably so in $\aca$) so that, since $\phi$ is $\Sigma^1_1$, $\atr$ proves the existence of $a$ such that $\phi(a) \land a \not\in\wo$. If $a$ were wellordered, then, $\mathcal{O}$ is $\Delta^1_1$ in $a$ and thus exists by $\DELTA^1_1{-}\mathsf{CA}_0$, from which the result follows easily, so assume that $a$ is illfounded. 

We take $M$ as the witness given by $\phi(a)$. The rest of the proof consists in verifying that $M$ is as desired. There are multiple ways of doing so; the following is one of them. Indeed, suppose towards a contradiction that there is $b \in \wo$ so that some initial segment of $b$ is not isomorphic to an initial segment of $a$. We claim that there is a least $c \prec b$, with this property. Indeed the two sets\begin{align*}
  &\phi(x) =  x \prec b \land \exists y\in \Ord^M \big( x \cong y\big); \\
  &\psi(x) =  x \prec b \land \forall y \in \mathrm{Ord}^{M} \ \big(\wo(y) \rightarrow y \prec x\big),
\end{align*} 
are both $\Sigma^1_1$. The fact that $\phi$ is $\Sigma^1_1$ is clear from the definition; the fact that $\psi$ is $\Sigma^1_1$ follows from an application of $\Sigma^1_1$-$\mathsf{AC}_0$, using that quantification over elements of $M$ is first-order.

Moreover, $\phi$ and $\psi$ are complementary. By $\Delta^1_1$--$\ca$, we can form the set $A$ of all $x$ satisfying $\phi(x)$. Using the order inherited from $b$, we can wellorder $A$ into a set whose order type is precisely the ordinal wellfounded part of $M$, say $c$. Since $b \in \wo$, we have $c \in \wo$ as well.
Then $L_{c}$ exists and is the wellfounded part of $M$, which, we claim, entails that $L_{c} \models \kp$, contradicting $c \in \wo$. 

To prove the claim, suppose that $L_{c}$ is inadmissible. By a result of Jensen (see~\cite[Lemma 2.11]{Jensen72}),
there is a $\Sigma_1$-definable total map from some $ d < c $ onto an unbounded subset of $c$,
say by a formula $\phi$ with a parameter $p \in L_c$. Let $w$ be an illfounded ordinal of
$M$. Let $Z$ be the set of all pairs $(x,y)$ such that $L^{M}_{w} \models \phi(p, x, y)$ and for no $y' <_{L^M} y$ do we have $L^{M}_{w} \models \phi(p, x, y')$. $Z$ exists by $\mathsf{ACA}_0$. Since $\phi$ defines a total map over $L_c$ and is $\Sigma_1$, it
follows that for every $x < d$ there is a unique $y \in L_c$ such that $(x,y) \in Z$. Thus,
$Z$ is an unbounded mapping from $d$ onto $c$ in $L_{w+1}^M$, and thus in $M$, contradicting the fact that $M$ satisfies the schema of \textsc{Foundation}.
\endproof

\subsection{Class random forcing over non-standard models}\label{SectClassForcing}
Eventually, we shall need to do random forcing over the model $M$ constructed in \S\ref{SectModelM}. This is an instance of class forcing over a non-standard model of a very weak set theory and thus requires some care.
For now, we only need to assume that $M$ is an $\omega$-model  of $\mathsf{BST} + V = L$. The fact that $M$ satisfies the full schema of \textsc{Foundation} will be crucial.

\begin{definition} 
Let $M$ be an $\omega$-model of $\mathsf{BST} + V = L$.
The \textit{random algebra} over $M$ is the partial order $\mathcal{B}_M$ consisting of all codes $c_A \in M$ of $\SIGMA^0_2$ sets such that $\lambda(A) > 0$ (equivalently, $M\models\lambda(A) > 0$).
Given $c_A, c_B \in \mathcal{B}_M$, we write $c_B \leq c_A$ ($c_B$ is \textit{stronger} than $c_A$) if $\lambda(B \setminus A) = 0$, i.e., $A$ contains $B$ up to a set of measure $0$. 
\end{definition}
Note that if $F$ is a filter on $\mathcal{B}_M$ then $F$ is necessarily closed under the equivalence relation given by $c_A \equiv c_{A'}$ if $c_A \leq c_{A'}$ and $c_{A'} \leq c_A$. We will identify equivalent conditions in $\mathcal{B}_M$ whenever this causes no confusion.
The random algebra is usually defined in terms of Borel sets. However,
by a theorem of Yu~\cite{Yu93}, $\atr$ proves that for all Borel $A$, there is a $\SIGMA^0_2$ set $B \subset A$ with $\lambda(B) = \lambda(A)$. 
Therefore, we restrict our focus to $\SIGMA^0_2$ sets as stated. This has the advantage of avoiding us having to deal with non-standard Borel codes from $M$. 

Our intention is now to force over $M$ in the set-theoretic sense. For genericity, we will employ $M$-definable genericity:

\begin{definition}
A filter $F$ on $\mathcal{B}_M$ is \textit{definably $M$-generic} if $F$ has nonempty intersection with every dense subset of $\mathcal{B}_M$ definable over $M$.
\end{definition}
By following the usual argument, one can see that (definable) genericity can be alternatively defined in terms of open dense subsets or maximal antichains.

We define the notion of a \textit{$\mathcal{B}_M$-name} inductively: these are sets of pairs $(c_A, \dot{x})$ in $M$, where $c_A \in \mathcal{B}_M$ and $\dot{x}$ is a $\mathcal{B}_M$-name. The set of names is called $M^{\mathcal{B}_M}$.

\begin{remark}
One could probably avoid the use of names altogether and force in a more recursion-theoretic manner (see e.g., Patey \cite{Pa24}), directly deciding infinitary properties about the generic real. Doing so would require the use of infinitary formulas indexed by possibly non-standard $M$-ordinals. It is not clear that one approach would be simpler than the other; the use of names in our construction is ultimately merely a matter of preference.
\end{remark}

The set of names is definable over $M$, so $M^{\mathcal{B}_M}$ is a definable class from the point of view of $M$. The fact that $M$ satisfies the schema of \textsc{Foundation} allows proving usual basic facts about $M^{\mathcal{B}_M}$. For instance, every set $x \in M$ has a \textit{canonical name} given by
\[\check{x} = \big\{(c_{[0,1]}, \check{y}) : y \in x\big\}.\] 
As usual, we abuse notation by identifying $\check{x}$ with $x$ if this causes no confusion.

We shall have to define the forcing relation over $M$. Letting $\alpha \in \Ord^M$, we write $c_A \Vdash^{\alpha} \dot{y} \in \dot{x}$ if $c_A, \dot{y}, \dot{x} \in L^M_{\alpha}$ and $L_{\alpha} \models c_A \Vdash_{\mathcal{B}_{L_{\alpha}}} \dot{y} \in \dot{x}$. Here, $\mathcal{B}_{L_{\alpha}}$ is an element of $M$, so
$\Vdash_{\mathcal{B}^{L_{\alpha}}}$, being an instance of set forcing, is defined as usual in $M$. Here, note the notation $\Vdash_{\mathbb{P}}$ emphasizing that $\mathbb{P}$ is the partial order under discussion.

\begin{lemma}
  If $c_A, \dot{y}, \dot{x} \in L^M_{\alpha}$ and $\bar{\alpha} > \alpha$, then \begin{align*}
    c_A \Vdash^{\alpha} \dot{y} \in \dot{x} \leftrightarrow c_A \Vdash^{\bar{\alpha}} \dot{y} \in \dot{x}.
  \end{align*}
\end{lemma}
\begin{proof}
This is proved by a straightforward induction on the rank of $\dot x$, using the fact that $M$ satisfies the schema of \textsc{Foundation}. See e.g., \cite[Lemma 14]{AgLu26}.
\end{proof}

\begin{definition}
The forcing relation $c_A \Vdash \phi(\dot{a})$ for a condition $c_A \in \mathcal{B}_M$, $\phi$ a formula of $L_{\mathrm{Set}}$, and $\dot{a}$ a $\mathcal{B}_M$--name is defined inductively the complexity of $\phi$. 
  \begin{enumerate}
  \item $c_A \Vdash \dot y \in \dot x$ holds if $c_A\Vdash^\alpha \dot y \in \dot x$ for all (equivalently, some) $\alpha > \text{rank} (\dot x), c_A$;
    \item $c_A \Vdash \phi \land \psi$ holds if $c_A \Vdash \phi$ and $c_A \Vdash \psi$;
    \item $c_A \Vdash \lnot \phi$ holds if $\forall c_B \leq c_A \ c_B \not\Vdash \phi$;
    \item $c_A \Vdash \exists x \ \phi(x)$ holds if 
    $$\forall c_B \leq c_A \ \exists c_D \leq c_B \ \exists \dot{a} \in M^{\mathcal{B}} \ c_D \Vdash \phi(\dot{a});$$
    \item $c_A \Vdash \forall x \ \phi(x)$ holds if $\forall \dot{a} \in M^{\mathcal{B}} \ c_A \Vdash \phi(\dot{a})$.
  \end{enumerate}
\end{definition}

When aiming at consistency results, one usually need not realize forcing extensions $M[G]$ and instead focuses on the forcing relation. Here, however, we will need to realize the forcing extensions within our model of $\mathsf{ATR}_0$. 
Usually this is done by induction on the rank, but our model of interest $M$ is non-standard. Another point of subtlety is the fact that the canonical name for the generic filter, $\{(c_A, \check{c_A}) : c_A \in \mathcal{B}_M\}$, is not an element of $M$.  This is problematic because we would like to prove the forcing theorem for formulas involving the generic. Finally, we face the problem that the forcing relation might not be definable, even when restricting the complexity of formulas, since forcing each bounded quantifier seemingly requires two alternating unbounded quantifiers. Nonetheless, when focusing on subformulas of a given $\phi$, the forcing relation is indeed definable by a formula with complexity determined by the logical complexity of $\phi$.
The following result summarizes how to overcome these difficulties, as well as the results we need concerning this generic extension. 

Below, recall that a subset $A$ of a model $M$ is \textit{amenable to $M$} if $A \cap x \in M$ for all $x \in M$.
\begin{lemma}[The forcing theorem for $\mathcal{B}_M$]\label{LemmaForcingTheorem}
Assume $\mathsf{ATR}_0$. Let $M$ be a model as above, possibly non-standard, and let $G\subset\mathcal{B}_M$ be definably $M$-generic. 
Then, there is a model $M[G]$ satisfying the following:
\begin{enumerate}
\item\label{LemmaForcingTheorem1} there is an embedding from $M$ into $M[G]$ with range definable in $M[G]$;
\item\label{LemmaForcingTheorem2} $\Ord^{M[G]} \cong \Ord^M$ under this embedding;
\item\label{LemmaForcingTheorem3} $G$ is amenable to $M[G]$;
\item\label{LemmaForcingTheorem4} Letting $g = \bigcap \{ A : c_A \in G\}$, we have $g \in \mathbb{R}\cap M[G]$ and $M[G] \models V = L[g]$;
\item\label{LemmaForcingTheorem5} for all $\phi \in L_{\mathrm{Set}}$, $\dot a \in M^{\mathcal{B}_M}$, we have
\[M[G] \models \phi(\dot a_G) \text{ if and only if } \exists p\in G\, (p\Vdash \phi(\dot a)).\] 
\end{enumerate} 
\end{lemma}
\proof
We realize $M[G]$ as the set $M^{\mathcal{B}_M} / G$ of $\mathcal{B}_M$-names modulo the equivalence relation given by 
\[\sigma \equiv \tau \leftrightarrow \exists p\in G\, (p\Vdash \sigma = \tau).\]
We denote by $\sigma_G$ the equivalence class of the name $\sigma$ and put 
\[\sigma_G \in \tau_G \leftrightarrow \exists p\in G\, (p\Vdash \sigma \in \tau).\]
Using \textsc{Foundation} in $M$, we prove by induction on the ranks of names that this membership relation is congruent with $\equiv$, so that $(M[G], \in)$ is well-defined. With this, clause \eqref{LemmaForcingTheorem5} is proved by an external induction on $\phi$ following the usual proof (see e.g., Jech \cite{Je03}), using the facts that the forcing relation $\Vdash$ restricted to subformulas of $\phi(\vec a)$ is definable, and that $G$ meets every subclass of $\mathcal{B}_M$ definable over $M$. Note that simple ``$M$-genericity'' would not suffice for this argument.
We find it convenient to prove \eqref{LemmaForcingTheorem4} next.

Below, we consider the following function $e$. The domain of $e$ consists of all elements of Cantor space $2^{\mathbb{N}}$ whose digits are not eventually constant, and its image is defined by
\[e: x\mapsto \sum_{i = 0}^\infty x(i)\cdot 2^{-i-1}.\]
Thus, $e$ is a homeomorphism between its domain and the set $(0,1) \setminus \mathbb{Z}[1/2]$ of dyadic irrational numbers between $0$ and $1$.

\begin{claim}\label{LemmaSigma02Imagee}
Let $e$ be the homeomorphism above and let $n\in\mathbb{N}, j \in \{0,1\}$. Then, $e[\{x \in 2^{\mathbb{N}} : x(n) = j\}]$ is open (in $(0,1)\setminus \mathbb{Z}[1/2]$).
\end{claim}
\proof
From the definition, we have $y \in e[\{x \in 2^{\mathbb{N}} : x(n) = j\}]$ if and only if for some $s \in 2^{n+1}$ with $s(n) = j$, we have $y \in e[\{x \in 2^{\mathbb{N}} : x\upharpoonright n+1 = s\}]$, i.e., if and only if 
\[\sum_{i = 0}^{n+1} s(i)\cdot 2^{-i-1} < y < \sum_{i = 0}^{n+1} s(i)\cdot 2^{-i-1} + 2^{-n-1},\]
as desired. This proves the claim.
\endproof

\begin{definition}
We define the canonical name $\check{g}$ for the generic random real by
\begin{align*}
    \check{g} = \big\{(c_{A_{n,j}}, (n,j)) : n \in \mathbb{N}, \ j \in \{0,1\} \big\}, 
  \end{align*}
where $A_{n,j} = e[\{x \in 2^{\mathbb{N}} : x(n) = j\}]$.
\end{definition}
Informally, the canonical name $\check{g}$ is defined so that the statement that the $n$th digit of $\check{g}$ is $j$ has as ``truth value'' (a code of) the set $A_{n,j}$.

Note that $A_{n,j}$ is open by Claim \ref{LemmaSigma02Imagee}, so $\check{g}$ is indeed a $\mathcal{B}_M$-name.  The advantage of this definition is that $\check{g}$ is computable, unlike $\check{G}$, which does not even belong to $M$.
We claim that
\begin{equation}\label{eqGenericReal}
g = \bigcap \{ A : c_A \in G\}, \quad \text{ where } g = (\check{g})_G.
\end{equation}
To prove the claim, we observe that by genericity, the set on the right-hand side will not be a dyadic rational, and thus will be in the range of the homeomorphism $e$ defined above.
Let $A_{n,j}$ be as in the definition of $\check{g}$.
 Thus, for each $n$, the set $\{c_{A_{n,0}}, c_{A_{n,1}}\}$ is a maximal antichain, so precisely one of $c_{A_{n,0}}$ or $c_{A_{n,1}}$ will belong to $G$. Since every $\SIGMA^0_2$ set of positive measure contains an open interval and $M$ 
conditions in $\mathcal{B}_M$ really code sets of positive measure by Lemma \ref{LemmaMeasureZeroSigma2}, it follows that a code $c_A \in M$ belongs to $G$ if and only if $A$ is contained in the intersection of finitely many $A_{n,j}$ containing $(\check{g})^G$, from which \eqref{eqGenericReal} follows. By \textsc{Foundation} in $M$, we see that $M$ contains names for $L_a[g]$ for all $a \in \Ord$.
By \eqref{eqGenericReal}, we have $L_a[G] \in L_{a+1}[g]$ for all $a \in \Ord^M$, so $M[G] \models V = L[g]$, completing the proof of \eqref{LemmaForcingTheorem4}, and yielding \eqref{LemmaForcingTheorem3} in the process.

We embed $M$ into $M[G]$ via the mapping $x\mapsto \check{x}$. Given a name $\sigma$ for an ordinal with $\sigma \in L_a^M$, one uses \textsc{Foundation} in $M$ to see that $\Vdash \sigma < a$, yielding \eqref{LemmaForcingTheorem2}.
Using \textsc{Foundation} in $M$, one shows that $\Vdash \check{y} \in \check{x}$ if and only if $y\in x$, and that for each $a \in \Ord^M$, we have $\Vdash L_{a}^{M[G]} = L_a^M$, proving \eqref{LemmaForcingTheorem1}. This completes the proof of the lemma.
\endproof

\begin{lemma}\label{LemmaFoundation}
Assume $\mathsf{ATR}_0$ and let $M$ be as above, $G\subset\mathcal{B}_M$ be definably $M$-generic, and $g$ be the generic real. Then, $L[g]$ satisfies full \textsc{Foundation}.
\end{lemma}
\proof
Fix a formula $\phi$ and a parameter $\sigma \in M^{\mathcal{B}_M}$. Let $c_A$ be a condition forcing that $\{x: \phi(\sigma, x)\}$ is nonempty.
Using the facts that the forcing relation restricted to subformulas of $\phi$ is definable and that $M$ satisfies full \textsc{Foundation}, we may find a condition $c_B\leq c_A$ and a name $\tau$ of minimal rank such that
\[c_B \Vdash \phi(\sigma,\tau).\]
Indeed the class of such conditions is dense below $c_A$, so we may assume $c_B \in G$ without loss of generality. Then, $\tau_G$ is as desired, for otherwise by Lemma \ref{LemmaForcingTheorem} we would have some $c_D \leq c_B$ in $G$ and names $\chi, \chi' \in M^{\mathcal{B}_M}$ such that $M[G] \models \phi(\sigma_G,\chi_G) \wedge \chi_G \in \tau_G$, $c_D \Vdash \chi = \chi'$ and $(c_D, \chi') \in \tau$, contradicting the $\in$-minimality of $\tau$.
\endproof

\subsection{$M$-random reals} 
Below, we assume $\mathsf{ATR}_0$ and let $M$ be as above. 

\begin{definition}
We say that a real $X$ is \textit{definably $M$--random} if $X \not \in A$ whenever $A$ is a $\SIGMA^0_2$ set of measure zero definable over $M$.
We denote by $R(M)$ the set of reals definably random  over $M$.
\end{definition}

Thus, definable $M$-randomness asserts that the real $X$ passes the collection of Martin-L\"of randomness tests definable over $M$.

\begin{lemma}
The set of real numbers $X$ which are definably not $M$--random has measure zero.
\end{lemma}
\proof
A real $x$ is not definably $M$--random if and only if it belongs to a measure zero $\SIGMA^0_2$ set definable over $M$. Using $\mathsf{ACA}^+_0$, and Lemma \ref{LemmaMeasureZeroSigma2}, the set of codes of such sets exists. 
Within $\aca$, the union of these sets is measurable, and by countable additivity of $\lambda$ (provable in $\mathsf{WWKL}_0$), this union has measure zero.
\endproof

\begin{remark}
  If $X$ is definably $M$--random, then $X$ is $\Delta^1_1$ random, by Lemma \ref{LemmaNonstandardOrdinals}. The converse is not true since $M$ could contain some non-standard code for Borel sets that are not $\Delta_1^1$. 
\end{remark}

\begin{definition}
We say that a real number $X$ is \textit{definably $\mathcal{B}_M$--generic} if 
\[X = \bigcap\Big\{A: c_A \in G\Big\}\]
for some $M$-definably generic $G\subset \mathcal{B}_M$.
\end{definition}

\begin{lemma}\label{LemmaFsigmaDense}
Let $M$ be as above; then, a real $X$ is definably $M$--random if and only if it is definably $\mathcal{B}_M$--generic.
\end{lemma}
\proof
This is proved within $\mathsf{ATR}_0$ the same way as it is in the case of $\mathsf{ZFC}$ (where one deals with full genericity), arguing similarly to the proof of \eqref{eqGenericReal} within Lemma \ref{LemmaForcingTheorem}; see e.g., Jech \cite{Je03}.
\endproof

Below, we recall that $\wo$ denotes the class of recursive wellorders.

\begin{lemma}\label{LemmaRandomRecursive}
Suppose $X$ is definably $M$-random and $Y \in \wo^X$. Then, $Y$ is isomorphic to some $Z \in \wo$.
\end{lemma}
\proof
Suppose otherwise and let $Y\in \wo^X$ be a counterexample. By comparability of wellorders, every $Z \in \wo$ is isomorphic to an initial segment of $Y$. Thus, a recursive linear ordering $Z$ is wellfounded if and only if there is an isomorphism from $Z$ to some initial segment of $Y$. It follows that $\wo$ is $\Delta^1_1(Y)$. By $\Delta^1_1{-}\ca$, $\mathcal{O}$ exists and thus $L_{\omega_1^{ck}}$ exists by $\atr$. By hypothesis, $X$ is $M$-random, so $M[X]$ is a generic extension of $M$ by $\mathcal{B}_{M}$, by Lemma \ref{LemmaNonstandardOrdinals} and Lemma \ref{LemmaForcingTheorem}. Since $X$ computes a wellordering of length $\omega_1^{ck} = \wfp(M)$, $M[X]$ satisfies a failure of $\PI_1$-\textsc{Foundation}, contradicting Lemma \ref{LemmaFoundation}.
\endproof

\subsection{Spector-Gandy representations}\label{SectSG}
We continue assuming $\mathsf{ATR}_0$ and let $M$ be as above.
\begin{lemma}\label{LemmaSG}
For every $\Pi^1_1$ formula $\psi(Y)$, 
there is a $\Sigma_1$ formula $\psi^*$ such that the following are equivalent for every real $Y$:
\begin{enumerate}
\item $\psi(Y)$, and
\item $\exists \alpha \in \wo^Y\, (L_\alpha[Y]\models\psi^*(Y))$.
\end{enumerate}
Moreover, the mapping $\psi\mapsto\psi^*$ is recursive.
\end{lemma}
\proof
This is a well-known result due to Spector and Gandy. By Simpson~\cite[VIII.3.27]{Simpson}, it is provable in $\mathsf{ATR}_0$.
\endproof

\begin{lemma}\label{LemmaSGRandom}
Let $A$ be a $\Pi^1_1$ set of reals. Then,
there is a $\Sigma_1$ formula $\varphi$ such that the following are equivalent for every real $X$ which is $M$-random:
\begin{enumerate}
\item $X\in A$, and
\item $\exists \alpha \in \wo\, (L_\alpha[X]\models\varphi(X))$.
\end{enumerate}
\end{lemma}
\proof
Immediate from Lemma \ref{LemmaSG} and Lemma \ref{LemmaRandomRecursive}.
\endproof

We also mention the following converse to the Spector-Gandy theorem:
\begin{lemma}\label{LemmaKleeneRC}
For every $\Sigma_1$ formula $\psi$, 
there is a $\Pi^1_1$ formula $\psi^*$ such that the following are equivalent for every real $Y$:
\begin{enumerate}
\item $\psi^*(Y)$, and
\item $\exists \alpha \in \wo^Y\, (L_\alpha[Y]\models\psi(Y))$.
\end{enumerate}
Moreover, the mapping $\psi\mapsto\psi^*$ is recursive.
\end{lemma}
\proof
This is a well-known result due to Kleene. For a proof that the lemma holds in $\atr$, see Simpson \cite[VIII.3.20]{Simpson}.
\endproof

\begin{corollary}\label{CorollaryDelta11Random}
The formula ``$Y$ is $\Delta^1_1$-random'' is $\Sigma^1_1$, provably in $\atr$.
\end{corollary}
\proof
This follows from Lemma \ref{LemmaKleeneRC}, as $Y$ is $\Delta^1_1$-random if and only if it is $L_\alpha$-random for every $\alpha\in \wo$.
\endproof

\subsection{Approximating sequences and nonstandard models}\label{SectionApprox}
From now on, we work in a model $V$ of $\mathsf{ATR}_0$ (or, equivalently, $\mathsf{ATR}_0^{\text{set}}$) satisfying $\SIGMA^1_1$-Induction. We fix a $\Pi^1_1$-definable class of reals $A$ and fix a formula $\varphi$ as in Lemma \ref{LemmaSGRandom}, for some model $M$ satisfying the conclusion of Lemma \ref{LemmaNonstandardOrdinals}. Note that $M$ is non-standard, so it may well disagree with $V$ concerning membership in $A$. Nonetheless, we shall use $M$ as an aid in establishing the Lebesgue-regularity of $A$. The proof will relativize easily to real parameters, thus yielding the conclusion for all co-analytic sets.

By Lemma \ref{LemmaInductionComprehension}, we have access to finite $\SIGMA^1_1$-Comprehension. 
In order to prove Lebesgue-regularity, we show that the intersection of
$A$ with a measure-one set (namely, $R(M)$) can be approximated in measure by open sets from outside and by compact sets from inside.
If so, we will have shown that 
\begin{equation}\label{eqRegularityRandom}
\lambda^*(A) = \lambda^*(A\cap R(M)) = \lambda_*(A\cap R(M)) =\lambda_*(A),
\end{equation}
as needed.

We write $M = L_a$, where $a = \Ord^M$ is non-standard, so $M$ is an $\omega$-model of $\mathsf{ACA}_0$. 
Recall that the notation $c_B$ denotes some Borel code for a Borel set $B$. Although this is ambiguous as a set might have different codes, we believe this will lead to no confusion.
Below, we may use the forcing relation $\Vdash^N_{\mathbb{P}}$ for the sake of clarity. This notation emphasizes that $\mathbb{P}$ is the partial order under consideration and that the relation is computed within the model $N$.
Observe that if $X$ is $M$-random, then
\begin{align}
\nonumber
X \in A
&\leftrightarrow \exists \alpha \in \wo\, \big(L_\alpha[X]\models\varphi(X)\big) 
	&\text{by Lemma \ref{LemmaSGRandom}}\\
&\leftrightarrow \exists \alpha \in \wo\, \exists c_B \in \mathcal{B}_{L_\alpha} \big(X\in B \wedge c_B\Vdash^{L_\alpha}_{\mathcal{B}_{L_\alpha}} \varphi(\check g)\big)
	&\text{by Lemmata \ref{LemmaFsigmaDense} and \ref{LemmaForcingTheorem}} \label{eqFsigma}
\end{align}

Below, we recall the G\"odel wellordering $<_L$. This is a uniformly definable $\Sigma_1$ wellordering of $L$ witnessing the Axiom of Global Choice in $L$. We use the notation $<_M$ for the corresponding wellordering of $M$ (defined by the same formula). The following definition will be a key part of the proof. 
\begin{definition}
Now, we say that a sequence $S = \{(\alpha_i, c_{B_i}, w_i) : i < b\}$ is an \textit{approximation} to $A$  \textit{with support} $\supp(S)$ and \textit{length} $b \in \Ord^M$ if
\[S \in L_{\supp(S)+1}^M\setminus L_{\supp(S)}^M\]
and for each $i <b$, $\alpha_i<a$ is the least $M$-ordinal for which there exist $c_{B_i},w_i \in L_{\alpha_i}^M$ satisfying the following conditions \eqref{DefApprox2}--\eqref{DefApprox4}:
\begin{enumerate}
\item\label{DefApprox2} $c_{B_i} \in L_{\alpha_i}^M$ is a code for a $\SIGMA^0_2$ set $B_i$ satisfying the following:
\begin{enumerate}
\item $\lambda(B_i) > 0$.
\item $c_{B_i}\Vdash^{L_{\alpha_i}}_{\mathcal{B}_{L_{\alpha_i}}}\varphi(\check g)$.
\item $B_i \cap \bigcup_{j<i} B_j = \varnothing$.
\end{enumerate} 
\item\label{DefApprox3} $c_{B_i}$ is $<_{M}$-least for which \eqref{DefApprox2} holds (with $\alpha_i$ fixed).
\item\label{DefApprox4} $w_i \in L_{\alpha_i}^M$ is the $<_{M}$-least witness to the $\Sigma_1$ formula $\theta^*$ given by Lemma \ref{LemmaSG}, where $\theta$ is the $\PI^1_1$ formula 
\[\forall Y\, \big( Y \in B_i \wedge \text{$Y$ is $\Delta^1_1$-random} \to Y \in A\big).\]
\end{enumerate}
\end{definition}
Here, we note that ``$Y$ is $\Delta^1_1$-random'' is $\Sigma^1_1$ by Corollary \ref{CorollaryDelta11Random}.
Note that the sets $B_i$ with codes $c_{B_i}$ in an approximation to $A$ might have different interpretations in $M$ and in $V$; condition \eqref{DefApprox4} ensures that $M$ thinks these sets are all contained in $A$, except possibly for a small set. Recall that by Lemma \ref{LemmaMeasureZeroSigma2}, the formula $\lambda(B_i) > 0$ is arithmetical and in particular absolute to $M$. By Lemma \ref{LemmaEmptyClosedSets}, the formula $B_i \cap \bigcup_{j<i} B_j = \varnothing$ is also arithmetical and thus also absolute to $M$; thus, it does not matter whether condition \eqref{DefApprox2} is checked in $V$ or in $M$. 

We will be interested in approximations to $A$ belonging to wellfounded initial segments of $M$. For these, the witnesses occurring in condition \eqref{DefApprox4} will really be witnesses for the truth of the $\PI^1_1$ formula $\theta$. Since we do not have access to $\PI^1_1{-}\mathsf{CA}_0$, we must allow for the possibility that approximations $S$ (as well as their supports) are illfounded, and attempt to make the most of the situation nonetheless.

\begin{claim}\label{ClaimUnique}
Suppose that $\beta \in \wo$. Then, there is at most one approximation $S$ to $A$ with length $\beta$ and any approximation with greater length extends $S$.
\end{claim}
\proof
Since $\beta \in \wo$, we have $\beta \in M$ (by choice of $M$).
By arithmetical transfinite induction on $i<\beta$ we see that there is at most one $M$-least ordinal $\alpha_i$ for which there are $c_{B_i}, w_i \in L_{\alpha_i}^M$ satisfying \eqref{DefApprox2}--\eqref{DefApprox4}. This transfinite induction is indeed arithmetical as all quantification ranges only over elements of $M$.
\endproof


\begin{definition}
Let $w \in M$. We say that $w$ is \textit{standard} if there is a wellorder $\alpha$ such that $w \in L_{\alpha+1}\setminus L_\alpha$, and \textit{nonstandard} otherwise.
\end{definition}

\begin{claim}\label{ClaimNonstandardApprox}
Suppose for all $\beta \in \wo$ there is an approximation with support greater than $\beta$. Then there is an approximation with nonstandard support $b$.
\end{claim}
\proof
Note that the definition of an approximation is internal to $M$, so if there is an approximation to $A$ with standard support $\beta$, then it must belong to $M$ and $M$ knows that it is an approximation to $A$. 
Since $M$ is a model of the full schema of \textsc{foundation}, an overspill argument yields such a $b$.
\endproof


\begin{claim}\label{ClaimDisjoint}
Suppose $S_b$ is an approximation of length $b$, $(\alpha_i, c_{B_i}, w_i) \in S_b$ and $w_i$ is nonstandard. 
Then, $B_i$ is disjoint from $A\cap R(M)$.
\end{claim}
\proof
This is because, by the definition of an approximation, the set coded by $c_{B_i}$ must be disjoint from $c_{B_j}$ for all smaller $j$, and in particular for all $j$ such that $w_j$ is standard.  By \eqref{eqFsigma}, however, the collection of $B_j$ for which $w_j$ is standard exhausts all of $A \cap R(M)$.
\endproof

Below, we write $B^*_\gamma$ for the union of the sets coded by $c_{B_i}$ for $i<\gamma$.
\smallskip

\text{Case I}: There is $\beta \in \wo$ such that all approximations have support $<\beta$. 
Then, using $\SIGMA^1_1{-}\mathsf{AC}$, we can collect the approximations into one set, $S_\infty$. Letting $B^*_\infty$ be the union of the $F_\sigma$ sets occurring in the approximations, it follows from \eqref{eqFsigma} that
\[A \cap R(M) = B^*_\infty \cap R(M),\]
from which \eqref{eqRegularityRandom} follows.
\smallskip

\text{Case II}: For every $\beta \in \wo$, there is an approximation of support $>\beta$. 
By Claim \ref{ClaimNonstandardApprox}, there is a nonstandard approximation $S_b$. 
Then, $B^*_b$ can be defined as above (both in $M$ and in $V$!). Being a subset of $[0,1]$ and a union of $\SIGMA^0_2$ sets, $\lambda(B^*_b)$ is defined and $\lambda(B^*_b)\leq 1$. Moreover, $B^*_b$ can be written in $M$ as the disjoint union 
\[B^*_b = \bigcup\Big\{B_i: i<^{M} b\Big\},\]
where $B_i$ really is a $\SIGMA^0_2$ set for each $i<^M b$ and where the sets really are disjoint, as $M$ is an $\omega$-model of $\mathsf{ACA}_0$. 
Thus, we have 
\begin{equation}\label{eqMmeasureb}
M \models \lambda(B^*_b) = \sum_{i<b} \lambda(B_i) \leq 1.
\end{equation}
Here is the idea for the remainder of the proof: we shall separate the $\SIGMA^0_2$ sets in the sequence according to their measure. Since the measure of their union is at most $1$ and the sets are disjoint, 
there can only be finitely many sets whose measure is not too small. By rearranging the sets according to their measure, we can represent the measure of $B^*_b$ as an $\mathbb{N}$-indexed sum of real numbers which converges.
We shall approximate the inner and outer measures of $A$ by finitely many of these $\SIGMA^0_2$ sets up to a margin of error given by the union of all sets in the sequence which have small enough measure. 
We now proceed with more details.
For each $n\in\mathbb{N}$, let
\[I_n = \Big\{i<^M b: M \models \lambda(B_i) > \frac{1}{2^n}\Big\}.\]
Thus, $|I_n| \leq 2^n$ for each $n$ and, letting $I_0 = \varnothing$, we have
\[M \models \lambda(B^*_b) = \sum_{n = 1}^\infty \lambda\bigg(\bigcup_{i\in I_n\setminus I_{n-1}} B_i\bigg).\]
 Since the tail series of any converging sum of real numbers converges to $0$, we have
\[M \models \lim_{n\to\infty} \lambda\bigg( \bigcup_{i\not\in I_n} B_i\bigg) = 0.\]
Towards establishing the Lebesgue-regularity of $A$, we now fix $\varepsilon>0$ and find $n\in\mathbb{N}$ large enough so that 
\[M\models \lambda\Big( \bigcup_{i\not\in I_n} B_i\Big) < \varepsilon.\]
Let $l < 2^n$ and $i_0, i_1, i_2, \hdots, i_l$ be the finitely many $M$-ordinals below $b$ such that 
\[M \models \lambda(B_{i_k}) > \frac{1}{2^n}\]
for each $k = 0, \hdots, l$. Let 
\begin{align*}
D^+ &= \bigcup \Big\{B_{i_k}: k \leq l \text{ and $(\alpha_{i_k}, c_{B_{i_k}}, w_{i_k})$ is standard}\Big\}\\
D^- &= \bigcup \Big\{B_{i_k}: k \leq l \text{ and $(\alpha_{i_k}, c_{B_{i_k}}, w_{i_k})$ is nonstandard}\Big\}.
\end{align*}
Appealing to $\SIGMA^1_1$-Induction and Lemma \ref{LemmaInductionComprehension}, we see that $D^+$ and $D^-$ exist as sets.
By Claim \ref{ClaimDisjoint}, $D^-$ is disjoint from $A\cap R(M)$. 
Thus, by \eqref{eqFsigma}, we have 
\[D^+ \subset A \cap R(M) \subset B^*_b\setminus D^-\]
and thus, recalling that $\lambda(R(M)) = 1$, we have
\[\lambda(D^+) = \lambda_*(D^+) \leq \lambda_*(A) \leq \lambda^*(A) \leq \lambda^*(B^*_b \setminus D^-) = \lambda(B^*_b \setminus D^-).\]
Since
\[(B^*_b \setminus D^-) \setminus D^+ = \bigcup_{i \not \in I_n} B_i,\]
it follows that
\[\lambda^*(A) - \lambda_*(A) \leq \lambda(B^*_b\setminus D^-) - \lambda(D^+) = \lambda\Big(\bigcup_{i\not\in I_n} B_i\Big) < \varepsilon,\]
which completes the proof.

\section{Conclusions}\label{SectConclusion}
Theorem \ref{TheoremRegular} is obtained by combining the proof in \S\ref{SectRegularity} with the reversal given by Lemma \ref{LemmaReversalInd}.
If one has access to $\PI^1_1$-Comprehension, then the argument of  \S\ref{SectRegularity} readily adapts to yield the existence of $\lambda(A)$, as one can directly separate the set of standard indices from nonstandard indices in an approximation to $A$.
(Note, however, that in this setting the proof could be simplified radically, taking $M = L_{\omega_1^{ck}}$, bypassing the reference to non-standard models, using $\Sigma_1$-Collection to avoid most subtleties with the forcing argument, and  having $A \cap R(M)$ represented as a union of approximations of length $\leq\omega_1^{ck}$.)
Theorem \ref{TheoremMeasurable} is obtained by combining this observation with the reversal given by Lemma \ref{LemmaReversalComp}.

Let us point out that while the proof in \S\ref{SectionApprox} makes use of the fact that the space $[0,1]$ is of finite measure, this fact is not required for the conclusion of the theorem.

\begin{theorem}
Suppose $\mathsf{ATR}_0 + \SIGMA^1_1{-}\mathsf{IND}$ holds. Then, every analytic set $A\subset\mathbb{R}$ is Lebesgue-regular.
\end{theorem}
\proof
Let $A\subset\mathbb{R}$ be $\Pi^1_1$. If $A$ has infinite inner measure, then there is nothing to prove. Otherwise, suppose that $\lambda_*(A) < \infty$. 
We consider a model $M$ and approximations $S$ to $A$ as in \S\ref{SectionApprox}. If all approximations $S$ have support $<\beta$ for some $\beta \in \wo$, then we complete the proof as in \textsc{Case I} in the proof in  \S\ref{SectionApprox}.

Otherwise we would like to argue as in \textsc{Case II} in the proof in  \S\ref{SectionApprox}. The  obstacle is that non-standard approximations might contain codes for $\SIGMA^0_2$ sets with arbitrarily large measure. Thus, we argue as follows. Fix $n^*$ such that $\lambda_*(A) < n^*$. Consider now the set
$N$ of all $n < n^*$ such that some standard approximation $S_\beta$ to $A$ satisfies 
\[n < \lambda(B^*_\beta),\]
where $B^*_\beta$ is, as in \S\ref{SectionApprox} the union of all sets $B_i$ with $c_{B_i}$ occurring in the approximation $S_\beta$. This is a $\PI^1_1$ set of numbers $<n^*$ which by $\SIGMA^1_1{-}\mathsf{IND}$ has a maximum element $n_0$. Thus, we know that
\[n_0 < \lambda_*(A) \leq n_0+1.\]
Now, let $S_b$ be a non-standard approximation, which exists by the case hypothesis. We may assume without loss of generality that $\lambda(B^*_b) \leq n_0+1$, for otherwise by Lemma \ref{LemmaFoundation} there is a least $M$-ordinal $a< b$ satisfying  $\lambda(B^*_a) > n_0+1$. Since every set occurring in an approximation has positive measure, $a$ must be a successor ordinal, and thus replacing $b$ by the predecessor of $a$, we have
\[\lambda(B^*_b) \leq n_0+1,\] 
as desired. Thus, we have 
\[M \models \lambda(B^*_b) = \sum_{i<b} \lambda(B_i) \leq n_0+1,\]
after which the proof proceeds exactly as in \textsc{Case II}  in  \S\ref{SectionApprox}, starting from \eqref{eqMmeasureb}.
\endproof

Similarly, we have:
\begin{theorem}
Suppose that $\PI^1_1{-}\mathsf{CA}_0$ holds. Then, every analytic set $A\subset\mathbb{R}$ is Lebesgue-measurable.
\end{theorem}

We finish with an application which does not make any reference to Reverse Mathematics. The following notion is familiar from computable analysis: a real $x\in\mathbb{R}$ is \textit{left-c.e.} if it is a limit of a recursively enumerable, increasing sequence of rational numbers. Below,  $\omega_1^x$ denotes the smallest ordinal not recursive in $x \in\mathbb{R}$.

\begin{definition}
A function $f:\mathbb{R}\to\mathbb{R}$ is \textit{left-coanalytic} if $\phi(x)$ is a limit of an increasing sequence of rational numbers which is $\Sigma_1$-definable over $L_{\omega_1^x}(x)$, uniformly.
\end{definition}

The proof of Theorem \ref{TheoremRegular} yields:
\begin{theorem}\label{TheoremComplexity}
Let $f$ be defined by $f(x) = \lambda(A)$ if $x$ codes a coanalytic set $A$ and $f(x) = 0$ otherwise. Then, $f$ is left-coanalytic.
\end{theorem}
\proof
Let $x$ be a code for a coanalytic set $A$.
Uniformly in $x$, the proof of Theorem \ref{TheoremRegular} (especifically, the construction in \S\ref{SectionApprox})  yields a sequence of approximations $\{S_\beta: \beta\leq \omega_1^x\}$ to $A$. From each $S_\beta$ one can compute the measure $\lambda(B^*_\beta)$ of the union of all $F_\sigma$ sets occurring in $S_\beta$. As these sets $B^*_\beta$ converge in measure to $A$ as $\beta$ approaches $\omega_1^x$, the result follows.
\endproof
Theorem \ref{TheoremComplexity} is best possible, as can be seen by an argument as in the proof of Lemma \ref{LemmaReversalComp}. Specifically, for all $x\in\mathbb{R}$, the set $A_x  = [0, \mathcal{O}^x] \in \Pi^1_1(x)$ satisfies $\lambda(A_x) = \mathcal{O}^x$.

\bibliographystyle{abbrv}
\bibliography{ReferencesJPA.bib}

\end{document}